\documentclass[12pt]{amsart}

\usepackage{amsmath}
\usepackage[normalem]{ulem}
\usepackage{amssymb,mathrsfs,mathcomp,amsfonts,longtable, hhline, textcomp}\usepackage{upgreek} 
\usepackage[matrix,arrow,curve]{xy}
\usepackage{hyperref}
\usepackage{appendix}
\usepackage{geometry}
\makeatletter
\newcommand*{\da@rightarrow}{\mathchar"0\hexnumber@\symAMSa 4B }
\newcommand*{\da@leftarrow}{\mathchar"0\hexnumber@\symAMSa 4C }
\newcommand*{\xdashrightarrow}[2][]{%
\mathrel{%
\mathpalette{\da@xarrow{#1}{#2}{}\da@rightarrow{\,}{}}{}%
}%
}
\newcommand{\xdashleftarrow}[2][]{%
\mathrel{%
\mathpalette{\da@xarrow{#1}{#2}\da@leftarrow{}{}{\,}}{}%
}%
}
\newcommand*{\da@xarrow}[7]{%
\sbox0{$\ifx#7\scriptstyle\scriptscriptstyle\else\scriptstyle\fi#5#1#6\m@th$}%
\sbox2{$\ifx#7\scriptstyle\scriptscriptstyle\else\scriptstyle\fi#5#2#6\m@th$}%
\sbox4{$#7\dabar@\m@th$}%
\dimen@=\wd0 %
\ifdim\wd2 >\dimen@
\dimen@=\wd2 %
\fi
\count@=2 %
\def\da@bars{\dabar@\dabar@}%
\@whiledim\count@\wd4<\dimen@\do{%
\advance\count@\@ne
\expandafter\def\expandafter\da@bars\expandafter{%
\da@bars
\dabar@ 
}%
}%
\mathrel{#3}%
\mathrel{%
\mathop{\da@bars}\limits
\ifx\\#1\\%
\else
_{\copy0}%
\fi
\ifx\\#2\\%
\else
^{\copy2}%
\fi
}%
\mathrel{#4}%
}
\makeatother

\newcommand{\comp}{\mathbin{\scriptstyle{\circ}}}
\newcommand{\B}{{\mathbf{B}}}
\newcommand{\Prm}{{\mathrm{P}}}

\newcommand{\QQ}{\mathbb{Q}}

\newcommand{\ZZ}{\mathbb{Z}}
\newcommand{\PP}{\mathbb{P}}
\newcommand{\OOO}{{\mathscr{O}}} 
 
\newcommand{\NNN}{{\mathscr{N}}} 
\newcommand{\MMM}{{\mathscr{M}}}

\newcommand{\qq}{\sim_{\scriptscriptstyle{\QQ}}}
\newcommand{\qW}{\operatorname{q_{\scriptscriptstyle{\mathrm{W}}}}}
\newcommand{\qQ}{\operatorname{q_{\scriptscriptstyle{\QQ}}}}

\newcommand{\Supp}{\operatorname{Supp}}

\newcommand{\Bs}{\operatorname{Bs}}
\newcommand{\Cl}{\operatorname{Cl}}
\newcommand{\Clt}[1]{\operatorname{Cl}(#1)_{\mathrm{T}}}
\newcommand{\rk}{\operatorname{rk}}
\newcommand{\Tors}{{\operatorname{Tors}}}
\newcommand{\dd}{\operatorname{d}}
\newcommand{\df}{\operatorname{df}}
\newcommand{\ct}{\operatorname{ct}}

\newcommand{\xref}[1]{\textup{\ref{#1}}}

\makeatletter
\@addtoreset{equation}{subsection}
\makeatother

\swapnumbers
\theoremstyle{plain}
\newtheorem{theorem}[subsection]{Theorem}
\newtheorem{stheorem}[equation]{Theorem}

\newtheorem{slemma}[equation]{Lemma}
\newtheorem{proposition}[subsection]{Proposition}
\newtheorem{sproposition}[equation]{Proposition}

\newtheorem{scorollary}[equation]{Corollary}

\newtheorem{sclaim}[equation]{Claim}

\theoremstyle{definition}

\newtheorem{sdefinition}[equation]{Definition}
\newtheorem{assumptions}[subsection]{Assumptions}

\newtheorem{sremark}[equation]{Remark}

\newcounter{NN}\numberwithin{NN}{section}

\title{Conic bundle structures on {$\QQ$}-Fano threefolds}
\author{Yuri Prokhorov}
\thanks{The paper was partially supported by the HSE University Basic Research Program.
}
\address{
\newline
\textnormal{Steklov Mathematical Institute of RAS,
8 Gubkina street, Moscow 119991, Russia.
}
\newline
\textnormal{
HSE University, 
Laboratory of Algebraic Geometry, 6 Usacheva str., Moscow, 119048, Russia.
}
\newline
\textnormal{
Faculty of Mechanics and Mathematics,
Moscow State University, Russia.
}
}
\email{prokhoro@mi-ras.ru}

\begin{document}

\maketitle

\section{Introduction}
A \textit{conic bundle} is a proper flat morphism $\pi: X \to Z$ of non-singular varieties such that it is of relative dimension 1 and the anticanonical divisor
$-K_X$ is relatively ample. 
We say that a variety $X$ has a conic bundle structure if there exists
a conic bundle $\pi: X'\to Z$ and a birational map $X \dashrightarrow X'$.

Varieties with conic bundle structure play a very important role in the 
birational classification of algebraic varieties of negative Kodaira dimension.
For example, any variety with rational curve fibration has a conic bundle structure \cite{Sarkisov:82e}. 
For these varieties there are well-developed techniques to solve rationality problems \cite{Sarkisov:82e}, \cite{Shokurov:Prym}, \cite{Iskovskikh:Duke}, \cite{P:rat-cb:e}.

Another important class of varieties of negative Kodaira dimension is
the class of \textit{$\QQ$-Fano varieties}. Recall that a projective variety $X$ is called $\QQ$-Fano if it has only terminal $\QQ$-factorial singularities,
the Picard number $\uprho(X)$ equals $1$, and the anticanonical class 
$-K_X$ is ample.
In fact, these two classes overlap. Moreover, $\QQ$-Fano varieties
with conic bundle structures form a large subclass in the class of all $\QQ$-Fano varieties. It is very important for birational geometry to 
investigate and describe those $\QQ$-Fano varieties that do not have conic bundle structures. This paper is an attempt to do it in dimension three. 

To formulate the main result we have to recall some definitions.
A normal $n$-dimensional variety $X$ with only canonical Gorenstein singularities is called a \textit{del Pezzo variety} if there exists an ample Cartier divisor $A$ such that $-K_X=(n-1)A$
(see \cite{Fujita:book}).
Then the intersection number $\dd(X):=A^n$ is called the \textit{degree} of $X$. 
We consider only those del Pezzo varieties that have at worst terminal singularities.
Del Pezzo varieties of degree $\le 2$ have special names:

\begin{enumerate}
\item 
a del Pezzo variety $X$ of degree $1$ is called \textit{double Veronese cone};
\item 
a del Pezzo variety of degree $2$ is called \textit{quartic double solid}.
\end{enumerate}

Recall that the Fano index $\qQ(X)$ of a $\QQ$-Fano variety is the maximal integer that divides the canonical class $K_X$ in the torsion free group $\Cl(X)/\Tors$
(see \ref{not:Q-Fano}). 
The following fact was proved in \cite{P:2019:rat:Q-Fano}:
\begin{theorem}
\label{thm:rat}
Let $X$ be a $\QQ$-Fano threefold with $\qQ(X)>7$.
Then $X$ is rational. In particular, $X$ has a lot of conic bundle structures.
\end{theorem}
For a normal $\QQ$-factorial projective variety $X$, denote 
\[
\df(X):= \max \big\{ \dim |M| \, \big| \, \text{$M$ is a Weil divisor such that $-(K_X+M)$ is ample} \big\}.
\]
Our main result is the following

\begin{theorem}
\label{thm:cbQf}
Let $X$ be a $\QQ$-Fano threefold with $\qQ(X)>1$.
Assume that $X$ has no conic bundle structures.
Then we have.
\begin{enumerate}
\item \label{thm:cbQf:0}
$\df(X)\le 3$.
\item 
\label{thm:cbQf-a}
If $\df(X)=3$, then $X$ is a smooth quartic double solid.
\item 
\label{thm:cbQf-b}
If $\df(X)=2$, then there exists a birational transformation $\Psi:X \dashrightarrow X'$ such that one of the following holds:
\begin{itemize}
\item 
$X'$ is a smooth quartic double solid, 
\item 
$X'$ is a smooth double Veronese cone and $\Psi$ is an isomorphism, or
\item 
$X'$ is a double Veronese cone with terminal $\QQ$-factorial singularities and 
$X'$ is singular.
\end{itemize}
\end{enumerate}
\end{theorem}

Note that a \textit{smooth} double Veronese cone has no conic bundle structures according to \cite{Grinenko:V1MFS}. Existence of conic bundle structures on smooth quartic double solids
and singular double Veronese cones is not known.

\subsection*{Acknowledgments.}
The author
would like to thank the referees for useful comments and
suggestions.

\section{Preliminaries}
\subsection{Notation and terminology}
In this paper we always assume that the ground field $\Bbbk$ is algebraically closed and of characteristic zero. 
Also, we use the standard notations and definitions of the Minimal Model Program (MMP).
When we say that a variety $X$ has terminal (resp. canonical) singularities,
this means that the singularities are not worse than that, in particular,
$X$ can be smooth.

For a variety $X$ with terminal singularities $\B(X)$ denotes the basket 
of its singularities (see \cite{Reid:YPG}).
Typically, when describing a basket we list only indices of singularities.
For example, $\B(X)=(r_1,r_2)$ means that $\B(X)$ contains two points
of types $\frac 1{r_1}(1,-1,b_1)$ and $\frac 1{r_2}(2,-2,b_2)$.
For any normal variety $X$, $\Cl(X)$ denotes the group of Weil divisors on $X$ modulo linear equivalence. By $\Clt{X}$ we denote the torsion subgroup of $\Cl(X)$. If $\MMM$ is a linear system, then $\Bs(\MMM)$ denotes its base locus.

\subsection{Singularities of linear systems}
\begin{sdefinition}
Let $X$ be a normal variety and let $\MMM$ be a (not necessarily complete) linear system of Weil divisors. We assume that $\dim(\MMM)>0$ and $\MMM$ has no fixed components. Let $\mu$ be non-negative rational number such that $K_X +\mu\MMM$ is
$\QQ$-Cartier. For a birational morphism $f:\tilde X\to X$, write 
\[
K_{\tilde X}+\mu \tilde \MMM = K_{X}+\mu \MMM +\sum e_i E_i,
\]
where $\tilde \MMM:=f^{-1}_*\MMM$ is the proper transform of $\MMM$ on $\tilde X$, $E_i$ are prime exceptional divisors, and $e_i$ are rational numbers called discrepancies.
Note that in this formula discrepancies are defined with respect to $ \mu M$, where $M\in \MMM$ is 
a general member \cite[Sect. 4]{Kollar95:pairs}.
We say that the pair $(X,\mu\MMM)$
\textit{canonical}, resp. \textit{terminal} if $e_i\ge 0$ (resp. $e_i>0$) for all $i$ and for all birational morphisms $f:\tilde X\to X$. 
\end{sdefinition}

\begin{slemma}[{\cite{Alexeev:ge}}, {\cite{Kollar95:pairs}}, {\cite{P:G-MMP}}]
\label{lemma:sing}
Let $(X,\MMM)$ be a pair, where $X$ is a threefold with at worst terminal 
singularities and $\MMM$ is a movable linear system on $X$. 
\begin{enumerate}
\item \label{lemma:sing:c}
If $(X,\MMM)$ is canonical, then a general member $S\in \MMM$ has only Du Val 
singularities and in a neighborhood of each point $P\in \Bs(\MMM)$ we have 
$\MMM\sim -K_X$, 
\item \label{lemma:sing:t}
If $(X,\MMM)$ is terminal, then a general member $S\in \MMM$ is a smooth 
surface contained in the smooth locus of $X$,
and $\dim \Bs (\MMM)\le 0$.
\end{enumerate}
\end{slemma}

\begin{stheorem}[{\cite{Alexeev:ge}}, {\cite{P:G-MMP}}]
Let $X$ be a variety with terminal $\QQ$-factorial singularities and let 
$\MMM$ be a linear system on $X$ without fixed components 
such that the pair $(X,\mu \MMM)$ is canonical for some $\mu\ge 0$.
Let 
\[
\varphi: (X,\mu \MMM)\dashrightarrow (X',\mu \MMM')
\]
be a step of 
$K_X+\mu \MMM$-MMP, that is, $\varphi$ is either a $K_X+\mu \MMM$-negative extremal divisorial contraction or a $K_X+\mu \MMM$-flip.
Then the variety $X'$ again has terminal $\QQ$-factorial singularities
and the pair $(X',\mu \MMM')$ is canonical.
If moreover $(X,\mu \MMM)$ is terminal, then so is $(X',\mu \MMM')$.
\end{stheorem}

\begin{scorollary}[\cite{Alexeev:ge}, \cite{BCHM}, \cite{P:G-MMP}]
\label{cor:terminal-model}
Let $X$ be a normal variety and let 
$\MMM$ be a linear system on $X$ without fixed components 
such that $K_X+\mu\MMM$ is $\QQ$-Cartier for some $\mu\ge 0$.
Then there exists a projective birational morphism 
\[
f:(\tilde X,\tilde \MMM)\longrightarrow (X,\MMM),
\]
where $\tilde \MMM:=f^{-1}_*\MMM$, such that the pair 
$(\tilde X,\mu\tilde \MMM)$ is terminal, $K_{\tilde X}+\mu\tilde \MMM$ 
is $f$-nef, and $\tilde X$ is $\QQ$-factorial. Moreover, 
\[
K_{\tilde X}+\mu\tilde \MMM\qq f^*(K_X+\mu\MMM)-\sum e_iE_i,
\]
where $E_i$ are prime exceptional divisors and $e_i\ge 0$ for all $i$.
\end{scorollary}
Such a pair $(\tilde X,\mu\tilde \MMM)$ is called \textit{ terminal $\QQ$-factorial 
model} of $(X,\mu\MMM)$.

\begin{scorollary}[\cite{Corti95:Sark}, \cite{Alexeev:ge}, \cite{P:G-MMP}]
\label{cor:can-blowup}
Let $X$ be a variety with terminal $\QQ$-factorial singularities and let 
$\MMM$ be a linear system on $X$ without fixed components 
such that the pair $(X, \mu\MMM)$ is canonical for some $\mu\ge 0$.
Then there exists a projective birational morphism 
\[
f:(\tilde X,\tilde \MMM)\longrightarrow (X,\MMM),
\]
where $\tilde \MMM:=f^{-1}_*\MMM$, such that $\tilde X$ is terminal $\QQ$-factorial, the pair 
$(\tilde X,\mu\tilde \MMM)$ is canonical, 
\[
K_{\tilde X}+\mu\tilde \MMM=f^*(K_X+\mu \MMM),
\]
and the exceptional locus of $f$ is a prime divisor. 
\end{scorollary}
Such a morphism is called an \textit{extremal log crepant blowup} of $(X, \mu\MMM)$.
\begin{scorollary}[\cite{BCHM}]
\label{cor:MMP}
Let $X$ be a normal projective variety and let 
$\MMM$ be a linear system on $X$ without fixed components 
such that $K_X+\mu\MMM$ is canonical for some $\mu\ge 0$.
Assume that $K_X+\mu\MMM$ is not pseudo-effective.
Then one can run $K_X+\mu\MMM$-MMP and end with a Mori fiber space.
\end{scorollary}

Recall also the following well-known result.

\begin{stheorem}[{\cite{Kawamata:Div-contr}}]
\label{theorem-Kawamata-blowup}
Let $(X \ni P)$ be a terminal quotient singularity of type
$\frac1r (1, a, r-a)$, let $ f \colon \tilde X \to X $ be a divisorial Mori contraction,
and let $E$ be the exceptional divisor.
Then $f(E)=P$, the contraction $f$ is a weighted blowup with weights $(1, a, r-a)$,
and the discrepancy of $E$ equals $a (E, X)=1/r$.
\end{stheorem}

\subsection{Del Pezzo varieties}

\begin{sproposition}[\cite{Fujita:book}, \cite{Shin1989}]
Let $X$ be a three-dimensional del Pezzo variety with terminal singularities.
Then 
\[
\dim \left|-\textstyle\frac12 K_X\right|=\dd(X)+1
\]
and a general member $M\in \left|-\textstyle\frac12 K_X\right|$ is 
a smooth del Pezzo surface of degree $\dd(X)$. Moreover, the pair 
$\left(X, \left|-\textstyle\frac12 K_X\right|\right)$ is terminal.
\end{sproposition}

\begin{stheorem}[\cite{Isk:Fano1e}, \cite{Fujita:book}, \cite{Shin1989}]
\label{thm:DP:cla}
Let $X$ be a three-dimensional del Pezzo variety with terminal singularities.
\begin{enumerate}
\item 
If $\dd(X)=1$ (i.e. $X$ is a double Veronese cone), then 
$X$ can be realized as a hypersurface of degree $6$ in the weighted projective space $\PP(1^3,2,3)$. In this case, the projection 
\[
\pi: X \longrightarrow \PP(1^3,2)
\]
is a double cover whose branch divisor has degree $6$ in $\PP(1^3,2)$. 
\item 
If $\dd(X)=2$ (i.e. $X$ is a quartic double solid), then 
$X$ can be realized as a hypersurface of degree $4$ in the weighted projective space $\PP(1^4,2)$. In this case, the projection 
\[
\pi: X \longrightarrow \PP(1^4)=\PP^3
\]
is a double cover whose branch divisor is a quartic hypersurface. 
\end{enumerate} 
\end{stheorem}

\subsection{$\QQ$-Fano varieties}
\label{not:Q-Fano}
Now, let $X$ be a $\QQ$-Fano variety.
It follows from the definition that $\Cl(X)$ is a finitely generated abelian group of rank 1. 
The numbers
\begin{eqnarray*}
\qQ(X)&:=&\max\{ t\in \ZZ \mid -K_X\qq tA,\ \text{$A$ is a Weil divisor}\} 
\\
\qW(X)&:=&\max\{ t\in \ZZ \mid -K_X\sim tA,\ \text{$A$ is a Weil divisor}\} 
\end{eqnarray*}
are called the \textit{Fano index} and \textit{Fano-Weil index}, respectively.
Clearly, $\qW(X)$ divides $\qQ(X)$ and $\qQ(X)=\qW(X)$ if $\Cl(X)$ is torsion free.
For a $\QQ$-Fano variety $X$ throughout this paper 
$A$ will denote a Weil divisor on $X$ such that $-K_X\qq \qQ(X) A$. If $\qQ(X)=\qW(X)$, we take $A$ so that $-K_X\sim \qQ(X) A$.
Since $\uprho(X)=1$ and $X$ is $\QQ$-factorial, for any Weil divisor $M$ on $X$ we can write $-K_X\qq \lambda M$. In this situation denote 
\[
\uplambda(X, |M|)=\uplambda(X, M):= \lambda. 
\]
Thus $\qQ(X)=\uplambda(X, A)$. 

\begin{stheorem}[{\cite{CampanaFlenner}}]
\label{thm:CF}
Let $X$ be a $\QQ$-Fano threefold and let $M$ be a smooth surface on $X$
with $\varkappa(M)=-\infty$.
Assume that $X$ is not rational. Then one 
of the following holds:
\begin{enumerate}
\item\label{thm:CF-WCI}
$X$ is a hypersurface of degree $6$ in $\PP(1^2,2^2,3)$, 
$M$ is a member of the linear system $|\OOO_X(2)|=\big|-\frac23 K_X\big|$, 
and so $\uplambda(X, M)=3/2$,
\item \label{thm:CF-DP}
$X$ is a del Pezzo threefold of degree $\dd(X)\le 3$, 
$M$ is a member of the linear system $\big|-\frac12 K_X\big|$, and $\uplambda(X, M)=2$.
\end{enumerate}
\end{stheorem}

Note that in \cite{CampanaFlenner} the surface $M$ is supposed to have ample normal bundle.
This is automatically satisfied in our case because $\uprho(X)=1$ by our $\QQ$-Fano assumption.

\begin{scorollary}
\label{lemma:DP}
Let $X$ be a $\QQ$-Fano threefold and let $\MMM$ be a linear system without fixed components on $X$ such that $\uplambda(X, \MMM)>1$ and the pair $(X,\MMM)$ is terminal.
Assume that $X$ has no conic bundle structures.
Then one of the following holds:
\begin{enumerate}
\item 
$X=X_6\subset \PP(1^3,2,3)$ is a double Veronese cone,
\item 
$X=X_4\subset \PP(1^4,2)$ is a \emph{smooth} quartic double solid.
\end{enumerate}
\end{scorollary}

\begin{proof}
By Lemma \ref{lemma:sing}\ref{lemma:sing:t} a general member $M\in \MMM$ is smooth 
and is contained in the smooth locus of $X$.
By the adjunction formula $M$ is a del Pezzo surface. In particular, $M$ is rational.
Thus we can apply Theorem~\ref{thm:CF}.

Assume that $X$ is such as in \ref{thm:CF}\ref{thm:CF-WCI}.
Let $\psi(x_1,x_1',x_2,x_2',x_3)=0$ be an equation of $X=X_6\subset \PP(1^2,2^2,3)$, 
where the subscript index of the variables $x_i$ and $x_j'$ is its degree. 
Since the singularities of $X$ are terminal, $\psi$ contains the term $x_3^2$.
Thus $\psi$ can be written in the form
\[
\psi=x_3^2 +\gamma_3(x_2,x_2') +\phi(x_1,x_1',x_2,x_2'),
\]
where $\gamma_3$ is a homogeneous polynomial of degree $3$ and $\phi$ does not contain cubic terms in $x_2,x_2'$. Clearly, $\gamma_3\neq 0$ (otherwise $X$ would be singular along the line $\{x_1=x_1'=x_3=0\}$). Thus after a linear coordinate change we may assume that $\gamma_3=x_2x_2'(x_2+x_2')$, $x_2x_2'^2$, or $x_2'^3$. Then the projection 
\[
X \dashrightarrow \PP(1,1,2),\qquad (x_1,x_1',x_2,x_2',x_3) \longmapsto (x_1,x_1',x_2') 
\]
is a rational curve fibration because
in the affine chart $x_2'=1$ the equation $\psi$ becomes 
quadratic in $x_3,x_2$. Hence $X$ has a conic bundle structure in this case.

Thus we may assume that $X$ is a del Pezzo threefold of degree $\dd(X)\le 3$
(see \ref{thm:CF}\ref{thm:CF-DP}).
If $\dd(X)= 3$, then $X=X_3\subset \PP^4$ is a cubic with terminal singularities
and $M$ is its smooth hyperplane section. In this case the projection $X \dashrightarrow\PP^2$ from a line $l\subset X$ is a rational curve fibration.

It remains to consider the case $\dd(X)= 2$, i.e. the case where
$X=X_4\subset \PP(1^4,2)$ is a quartic double solid.
Suppose that $X$ has a singular point, say $P\in X$. 
Let $\psi(x_0,x_1,x_2,x_3,y)=0$ be an equation of $X=X_4\subset \PP(1^4,2)$, where $\deg (x_i)=1$ and $\deg(y)=2$. 
Since the singularities of $X$ are terminal Gorenstein, $\psi$ contains the term $y^2$. By an obvious coordinate change we may assume that $\psi$ has the form
\[
\psi=y^2+\phi(x_0,x_1,x_2,x_3), 
\]
where $\deg(\phi)=4$. Clearly, $P$ is contained in the hyperplane $\{y=0\}$. 
Hence, by a linear coordinate change we may assume that $P=(1,0,0,0,0)$ and so
\[
\psi=y^2+\phi_4(x_1,x_2,x_3)+x_0\phi_3(x_1,x_2,x_3)+x_0^2\phi_2(x_1,x_2,x_3), 
\]
where $\deg (\phi_i)=i$.
As above, the projection 
\[
X \dashrightarrow \PP(1,1,1)=\PP^2,\qquad (x_0,x_1,x_2,x_3,y) \longmapsto (x_1,x_2,x_3) 
\]
is a rational curve fibration. Hence $X$ has a conic bundle structure.
\end{proof}

\section{Construction.}
\label{sect-constr}
We recall the construction used in the papers \cite{P:degQ-Fano-e},
\cite{P:2010:QFano}, \cite{P:2013-fano}, \cite{P-Reid}.

\begin{assumptions}
Let $X$ be a $\QQ$-Fano threefold.
Consider a linear system $\MMM$ on $X$ without fixed components.
Let $c:=\operatorname {ct} (X, \MMM)$ be the canonical threshold of the pair $(X, \MMM)$ \cite{Corti95:Sark}.
Assume that $\uplambda(X,\MMM)>c$.
\end{assumptions}

\begin{slemma}[see {\cite[Lemma 4.2]{P:2010:QFano}}]
\label{lemma-cthreshold}
Let $P \in X$ be a point of index $r>1$. Assume that
$\mathscr M \sim -tK_X$ near $P$, where $0<t<r$. Then $\ct(X,\MMM) \le 1/t$
and $\beta\ge t\alpha$.
\end{slemma}

Consider an extremal log crepant blowup $f: \tilde X \to X$ with respect to $K_X+c \MMM$
(Corollary~\ref{cor:can-blowup}).
Let $E$ be the exceptional divisor.
Recall that
$\tilde X$ has only terminal $\QQ$-factorial singularities.
We can write
\begin{equation} \label{equation-1}
\begin{array}{lll}
K_{\tilde X} &\qq & f^*K_X+\alpha E,
\\[2pt]
\tilde\MMM &\qq & f^*\MMM- \beta E.
\end{array}
\end{equation}
where $\alpha,\, \beta \in \QQ_{>0}$.
Then $c=\alpha/\beta$ and
\[
K_{\tilde X} +c \tilde \MMM \qq f^*(K_X+c \MMM).
\]
Since $\uplambda(X,\MMM)>c$, the divisor $-( K_{\tilde X} +c \tilde \MMM)$ is nef.
Take $\lambda=\uplambda(X,\MMM)$. Then 
\[
K_{\tilde X} +\lambda\tilde \MMM \qq f^*(K_X+\lambda \MMM)+\alpha(1-\lambda/c) E
\qq \alpha(1-\lambda/c) E.
\]
Put $\delta:=\alpha(\lambda/c-1)$. Then 
\begin{equation}
\label{eq:constr:Klambda}
K_{\tilde X} +\lambda \tilde \MMM+\delta E\qq0,\quad \text{where $\delta >0$.}
\end{equation} 
Run the MMP with respect to $K_{\tilde X}+c \tilde\MMM$.
We obtain the following diagram (Sarkisov link)
\begin{equation}
\label{diagram-main}
\vcenter{\xymatrix{
&\tilde X\ar[dl]_{f} \ar@{-->}[r]^\chi & \bar X\ar[dr]^{\bar f}&
\\
X\ar@{-->}[rrr]&&&\hat X 
}}
\end{equation}
Here the varieties $\tilde X$ and $\bar X$
have only terminal $\QQ$-factorial singularities, $\chi$ is a composition of $K_{\tilde X}+c \tilde\MMM$-log flips,
$\uprho (\tilde X)=\uprho (\bar X)=2$, and
$\bar{f}: \bar{X}\to \hat{X}$ is an extremal $K_{\bar{X}}$-negative Mori contraction.
In particular, $\rk \Cl(\hat{X})=1$.

In what follows, for a divisor (or a linear system) $N$ on $X$
by $\tilde N$ and $\bar N$ we denote
proper transforms of $N$
on $\tilde X$ and $\bar{X}$ respectively. If $\bar f$ is birational, then 
$\hat N$ denotes the 
proper transform of $N$
on $\hat X$.
Apply $\chi_*$ to \eqref{eq:constr:Klambda}:
\begin{equation}
\label{eq:constr:Klambda-1}
K_{\bar X} +\lambda \bar \MMM+\delta \bar E\qq0.
\end{equation}

\begin{slemma}[cf. {\cite[Lemma~4.7]{P:2019:rat:Q-Fano}}]
\label{lemma:MFS}
If $\MMM$ is not composed of a pencil and $\bar f$ is not birational, then $X$ has a conic bundle structure.
\end{slemma}

\begin{proof}
If $\dim(\hat{X})=2$, then $\bar f$ is a $\QQ$-conic bundle and we are done.
Assume that $\dim(\hat{X})=1$. Then $\bar f$ is a del Pezzo fibration and $\hat{X}\simeq 
\PP^1$.
By our assumption $\dim(\MMM)\ge 2$ and a general member of $\MMM$ is irreducible. Hence $\MMM$ 
is not $\bar f$-horizontal, i.e. $\MMM$ is not a pull-back of a linear system 
on $\hat{X}$. Since the divisor $\bar E$ is not movable, it is not contained in fibers
and so $\bar E$ is $\bar f$-ample.
For a general fiber $F$ by the adjunction formula and \eqref{eq:constr:Klambda-1} we have 
\[
-K_F=-K_{\bar X}|_F=\lambda \bar \MMM|_F+\delta \bar E|_F, 
\]
where both $\bar \MMM|_F$ and $\bar E|_F$ are integral ample divisors. Since $\lambda >1$,
the surface $F$ is isomorphic either $\PP^2$ or $\PP^1\times \PP^1$. In this case,
$\bar X$ is rational and so it has a lot of conic bundle structures. 
\end{proof}

\subsection{}
Assume that the contraction $\bar{f}$ is birational.
Then $\hat{X}$ is a $\QQ $-Fano threefold.
Denote
$\hat{q}:=\qQ (\hat{X})$.
\begin{slemma}
\label{lemma:Econtr}
If the contraction $\bar{f}$ is birational, then the divisor $\bar E$ is not contracted by $\bar f$.
\end{slemma}

\begin{proof}
Indeed, otherwise the map $\bar f\comp\chi\comp f^{-1}:X \dashrightarrow \hat X$ 
would be an isomorphism in codimension one.
Hence it is an isomorphism. On the other hand, 
the number of $K_{\hat X}+c\hat \MMM$-crepant divisors on $\hat X$ is strictly less than the number of $K_{X}+c\MMM$-crepant divisors on $X$, a contraction.
\end{proof}

It follows from \eqref{eq:constr:Klambda-1} that
\begin{equation*}
K_{\hat X} +\lambda \hat \MMM+\delta \hat E\qq0.
\end{equation*} 
Therefore, 
\begin{equation}
\label{eq:constr:Klambda3}
\uplambda(\hat X, \hat \MMM)> \uplambda (X,\MMM).
\end{equation} 

\subsection{}
Let $\NNN_k$ be a non-empty linear system on $X$ such that
$q\NNN_k\sim k(-K_X)$ 
(here we allow $\NNN_k$ have fixed components).
Let $\Theta$ be an ample Weil divisor on $\hat{X}$
generating $\Cl(\hat{X})/\Clt{\hat{X}}$.
We can write
\[
\hat{E} \qq e\Theta,\qquad \hat\NNN_k \qq s_k\Theta,
\]
where $e\in \ZZ_{>0}$ by Lemma~\ref{lemma:Econtr} and $s_k \in \ZZ_{\ge0}$. 
Thus $s_k=0$ if and only if 
$\dim (\NNN_k)=0$ and a unique element $\bar N$ of $\bar\NNN_k$ is $\bar{f}$-exceptional.
As in \eqref {equation-1}, we write
\begin{equation} \label{equation-12}
\tilde \NNN_k\qq f^*\NNN_k- \beta_k E.
\end{equation}
The relations \eqref {equation-1} and \eqref {equation-12} give us
\begin{equation*} \label{equation-12-divisors}
k K_{\tilde{X}}+q \tilde\NNN_k \sim -(q \beta_k-k \alpha) E,
\end{equation*}
where $q \beta_k-k \alpha$ is an integer. From this we obtain
\begin{equation} \label{equation-main}
k \hat{q}=q s_k+(q \beta_k-k \alpha) e.
\end{equation}

\subsection{}
We need certain information on linear systems on $\QQ$-Fano threefolds 
of large Fano index. Most of these facts are contained in the Graded Ring Database \cite{GRD}
or can be obtained by direct computations using algorithms described in 
\cite{Suzuki-2004}, \cite{Brown-Suzuki-2007j}, \cite{P:2010:QFano}, \cite{P:2019:rat:Q-Fano}.

\begin{sproposition}
\label{prop:search0} 
Let $X$ be a $\QQ$-Fano threefold and let $M$ be a Weil divisor on $X$ such that 
$\dim |M|\ge 2$ and $\uplambda(X,M)>1$.
Then the linear system $|M|$ is not composed of a pencil.
\end{sproposition}

\begin{proof}
Assume the contrary, that is, $|M|=D+m|L|$, where $D$ is the fixed part of $|M|$ and $|L|$ is a pencil without fixed components. Then $\dim |M|=m\ge 2$. Replacing $|M|$ with $m |L|$ we may assume that $D=0$. 
Clearly, $\dim |2L|=2$ and $\qQ(X)\ge 3$ in our case.
Now, running computer search for 
$\QQ$-Fano threefolds with $\dim |L|=1$ and $\dim |2L|=2$, we get a contradiction.
Interested readers can find a simple PARI/GP code \cite{PARI} 
on the author's webpage \url{https://homepage.mi-ras.ru/~prokhoro/programs/pencil.gp}.
\end{proof}

As above, using computer search one can find that there are 30 possible Hilbert series 
of $\QQ$-Fano threefolds such that $\qQ(X)\ge 3$, $\Clt{X}=0$, and $\df(X)=3$.
As a consequence of this list one obtains.

\begin{sproposition}
\label{prop:search1} 
Let $X$ be a $\QQ$-Fano threefold with $\qQ(X)\ge 3$, $\Cl(X)\simeq \ZZ$ and $\df(X)=3$. Then the following assertions hold.
\begin{enumerate}
\item \label{prop:search1:q=6}
$\qQ(X)\le 7$, $\qQ(X)\neq 6$, and the basket $\B(X)$ contains at 
most one point of index $\ge 8$.

\item
\label{prop:search1:q=7} 
If $\qQ(X)=7$, then $|A|=\varnothing$ and $\dim |2A|=\dim |3A|=0$.

\item\label{prop:search1:q=5} \label{prop:search1:q=5a} 
If $\qQ(X)=5$, then $\dim |2A|=0$ and $\dim |3A|= 1$.

\item\label{prop:search1:q=4} 
If $\qQ(X)=4$, then $\dim |A|\le 0$ and $1\le \dim |2A|\le 2$.
If moreover, $\dim |2A|=2$, then $\dim |A|= 0$, $A^3=2/11$ and $\B(X)=(11)$.
\end{enumerate}
\end{sproposition}

Similarly, we have.

\begin{sproposition}
\label{prop:search2} 
Let $X$ be a $\QQ$-Fano threefold with $\qQ(X)\ge 3$, $\df(X)=3$ and $\Clt{X}\simeq \ZZ/n\ZZ$, where $n>1$.
Let $T$ be a generator of $\Clt{X}$. Then the following assertions hold.
\begin{enumerate}
\item \label{prop:search2:q=6}
$\qQ(X)\le 5$, 

\item
\label{prop:search2:q=5}
If $\qQ(X)=5$, then $\qQ(X)=5$, $n=2$, $\B=(4^2, 12)$, $A^3=1/12$. Moreover,
$\dim |A|=0$,
$|A+T|=\varnothing$,
$\dim |2A|=\dim |2A+T|=0$,
$\dim |3A|=\dim |3A+T|=1$.

\item
\label{prop:search2:q=4}
Assume that $\qQ(X)=4$ and $\df(X)=3$. 
Then $n\in \{2,\, 5\}$ and for a suitable choice of $A$ we have $|A|=\varnothing$, $\dim |2A|=1$, and $\dim |3A+kT|=3$ for any $k$. Moreover, 
\begin{enumerate}
\item 
if $n=2$, then $\dim |A+T|=0$ and $0\le \dim |2A+T|\le 1$;
\item 
if $n=5$, then $\dim |A+kT|=0$ and $\dim |2A+kT|=1$ for $k\not \equiv 0\mod 5$.
\end{enumerate}

\item
\label{prop:search2:q=3}
Assume that $\qQ(X)=3$ and $\df(X)=3$. Then $n\in \{2,\, 3\}$ and $|A+\Lambda|\neq \varnothing$ for any
$\Lambda\in \Clt{X}$. Moreover, if $n=2$ and $\dim |2A|<3$, then $A^3=15/28$ and $\B=(2, 4, 14)$, and $|A|\neq \varnothing$.

\end{enumerate} 
\end{sproposition}

\section{Proof of Theorem~\ref{thm:cbQf}}
\begin{assumptions}
\label{ass0}
Let $X$ be a $\QQ$-Fano threefold with $\qQ(X)>1$ and $\df(X)\ge 2$ and let $\MMM$ be a 
linear system without fixed components such that $\dim(\MMM)=\df(X)$ and $\uplambda(X,\MMM)>1$.
Assume that $X$ has no conic bundle structures.
\end{assumptions}

\begin{sclaim}
\label{claim:bir-map}
There exists a sequence of links of the form \eqref{diagram-main}
\begin{equation}
\label{eq-seq}
\Phi: X= X^{(1)}\xdashrightarrow{\ \Psi_1\ } X^{(2)} \xdashrightarrow{\ \Psi_2\ } \cdots \xdashrightarrow{\ \Psi_{n-1}\ } X^{(n)}=X'
\end{equation} 
where
each $X^{(i)}$ is a $\QQ$-Fano threefold, 
\[
\uplambda\left(X^{(i+1)},\MMM^{(i+1)}\right)\ge \uplambda\left(X^{(i)},\MMM^{(i)}\right).
\]
for $\MMM^{(i+1)}=\Psi_{i*}\MMM^{(i)}$, 
and the pair 
$\left(X', \MMM'=\MMM^{(n)}\right)$ is terminal.
\end{sclaim}

\begin{proof}
If the pair $(X,\MMM)$ is terminal, we are done.
Thus we may assume that $c:=\ct(X,\MMM)\le 1$.
By Proposition~\ref{prop:search0} the linear system $\MMM$ is not composed of a pencil.
Apply the construction~\eqref{diagram-main}
to $(X,\MMM)$. 
By Lemma~\ref{lemma:MFS} the contraction $\bar f$ is birational and so
we obtain a new pair $(\hat X,\hat \MMM)$, where $\hat X$ is a $\QQ$-Fano threefold and $\hat \MMM:=\Psi_*\MMM$ is a 
linear system without fixed components such that $\uplambda(\hat X,\hat \MMM)>\uplambda(X,\MMM)$ (see \eqref{eq:constr:Klambda3}).
If the pair $(\hat X,\hat \MMM)$ is terminal, we are done. Otherwise 
we can repeat the process applying the construction~\eqref{diagram-main}
to $(\hat X,\hat \MMM)$ and continue.
We get a sequence of pairs 
\begin{equation}
\label{eq-seq1}
(X,\MMM)= (X^{(1)},\MMM^{(1)})\dashrightarrow (X^{(2)},\MMM^{(2)}) \dashrightarrow \cdots \dashrightarrow (X^{(n)},\MMM^{(n)})\dashrightarrow \cdots
\end{equation} 
Since the set of all $\QQ$-Fano threefolds is bounded \cite{Kawamata:bF}, the process terminates with a pair $(X',\MMM')=(X^{(n)},\MMM^{(n)})$ having terminal singularities. 
\end{proof}

By Corollary~\ref{lemma:DP} 
\ $X'$ is a del Pezzo threefold of degree $\dd(X')\le 2$ 
and $\MMM'=\Psi_*\MMM\subset |-\frac12 K_{X'}|$. In particular, 
$\dim(\MMM)=\dim(\MMM')\le 3$. This proves \ref{thm:cbQf}\ref{thm:cbQf:0}.

If 
$X'=X_6'\subset \PP(1^3,2,3)$ is a double Veronese cone, then $\dim(\MMM)=\dim(\MMM')=2$.
If, furthermore, 
$X'$ is smooth, then according to \cite{Grinenko:V1MFS}, any birational model of $X'$ which is a Mori fiber space non-isomorphic to $X'$ is a degree $1$ del Pezzo fibration over $\PP^1$. In particular, this implies that $X\simeq X'$.
If $X'$ is a quartic double solid, then it must be smooth (see Corollary~\ref{lemma:DP}). This proves \ref{thm:cbQf}\ref{thm:cbQf-b}.

It remains to prove the assertion of \ref{thm:cbQf}\ref{thm:cbQf-a}, i.e. 
$\Psi$ is an isomorphism in the case $\dim(\MMM)=3$. 
For this purpose we can consider the last step of \eqref{eq-seq}.
We show that $X^{(n)}\simeq X^{(n-1)}$. 

\begin{assumptions}
\label{ass}
Thus we assume that $\dim(\MMM)=3$. We put $X=X^{(n-1)}$ and $\hat X=X^{(n)}$.
Recall that $\hat X$ is a smooth quartic double solid in our case.
We are going to use all the notation of Sect.~\ref {sect-constr}.
Since $(X,\MMM)$ is not terminal, $c=\ct(X,\MMM)\le 1$.
Recall also the commutative diagram \eqref{diagram-main}:
\begin{equation}
\label{diagram-main-1}
\vcenter{\xymatrix{
&\tilde X\ar[dl]_{f} \ar@{-->}[r]^\chi & \bar X\ar[dr]!<-3em,0em>^{\bar f}&
\\
X\ar@{-->}[rrr]&&&\hat X =\hat X_6 \subset \PP(1^4,2)
}}
\end{equation}
\end{assumptions}

Denote by $\bar F$ the
$\bar{f}$-exceptional divisor and by
$\tilde F \subset \tilde X$ its proper transform.
By Lemma~\ref{lemma:Econtr}\ $\bar F\neq \bar E$. Hence $F:=f(\tilde F)$
is a prime divisor on $X$.
Write 
\[
\MMM\qq bA,\qquad F\qq dA,\qquad b,\, d\in \ZZ_{>0}. 
\]
Thus $\lambda:=\uplambda(X,\MMM)=q/b>1$. 
Clearly,
\begin{equation*}
\label{eq:QDS:discr-0}
K_{\bar X}\sim \bar f^* K_{\hat X}+ a \bar F,
\end{equation*} 
where $a$ is the discrepancy of $\bar F$. Since $\hat X$ is smooth, $a$ is a positive integer. 

\begin{sremark}
\label{sremark:discr}
\begin{itemize}
\item 
If $\bar f(\bar F)$ is a point, then by \cite{Kawakita:smooth} $\bar f$ is a weighted blowup of $\bar f(\bar F)\in \hat X$ with weights $(1,w_1,w_2)$, where $\gcd(w_1,w_2)=1$. In particular, $a=w_1+w_2\ge 2$. Moreover, $a=2$ if and only if $\bar f$ is a usual blowup. 
\item 
If $\bar f(\bar F)$ is a curve, then the contraction $\bar f$ is the usual blowup at its general point. In particular, $a=1$.
\end{itemize}
\end{sremark}

\begin{sclaim}
\label{claim:QDS:eq}
$2b=q+ a d$. In particular, $q<2b$.
\end{sclaim} 
\begin{proof}
Since $\Bs(\hat\MMM)=\varnothing$, we have 
\begin{equation*}
K_{\bar X}+2\bar \MMM=\bar f^* (K_{\hat X}+2\hat \MMM) +a \bar F\sim a \bar F.
\end{equation*} 
Taking the pushforward of this relation to $X$, we get the desired equality.
\end{proof}

\begin{sclaim}
\label{claim:tors0}
$\gcd(b,d)=1$ and
$\Clt{X}$ is a cyclic group of order $e/d$.
\end{sclaim}
\begin{proof}
The group $\Cl(\bar X)$ is generated by the classes of $\bar \MMM$ and $\bar E$:
\[
\Cl(\bar X)=\ZZ\cdot [\bar \MMM]\oplus \ZZ\cdot [\bar F].
\]
Hence, $\Cl(\tilde X)=\ZZ\cdot [\tilde \MMM]\oplus \ZZ\cdot [\tilde F]$ and we can write 
\begin{equation}
\label{eq:EuMvF}
E\sim u \tilde \MMM+v \tilde F.
\end{equation} 
Since $E$ is contracted by $f$, we have $\Cl(X)=\Cl(\tilde X)/\ZZ\cdot [E]$ and so 
$\Cl(X)=\ZZ\oplus \ZZ/n\ZZ$, where $n=\gcd(u,v)$.
Taking pushforward of \eqref{eq:EuMvF} to $X$ and $\hat X$ we get $ub+vd=0$ and $e=u$, respectively. Since the classes of $\MMM$ and $F$ generate $\Cl(X)$, we have $\gcd(b,d)=1$. Hence, $v=-nb$ and $e=nd$.
\end{proof}

\begin{sclaim}
\label{claim:MN-curve}
Assume that $\MMM\ni tN$, where $N$ is an (effective) Weil divisor and $t\ge 2$.
Then $e=d=1$, 
$\Cl(X)$ is torsion free, $\Supp(N)=F$, $\bar f(\bar F)$ is a point, and $a\ge 2$. 
\end{sclaim}
\begin{proof}
We can write $\tilde \MMM\sim t\tilde N+\delta E$, where $\delta$ is an non-negative integer.
Hence, $\hat \MMM\sim t\bar f_*\bar N+\delta\hat E$.
Since the class of $\hat \MMM$ is not divisible in $\Cl(\hat X)$,
we have $\Supp(\bar N)=\bar F$ and so $\delta=1$ and $\hat E\sim \MMM$, i.e. $e=1$.
By Claim~\ref{claim:tors0} $d=1$ and
$\Cl(X)$ is torsion free.
Further,
\[
\bar E\sim \bar \MMM-t\bar F\sim \bar f^*\hat E-t\bar F,
\]
where $t\ge 2$. If $\bar f(\bar F)$ is a curve, then $\hat E$ must be singular 
along $\bar f(\bar F)$. This contradicts \cite{Furushima-Tada:nnDP}. Hence $\bar f(\bar F)$ is a point and $a\ge 2$ by Remark~\ref{sremark:discr}.
\end{proof}

\begin{sclaim}
\label{claim:MN-new}
Assume that $\MMM\ni N_1+N_2$, where $N_1$ and $N_2$ are effective non-zero Weil divisors without common components. If $F\neq \Supp(N_1)$, then $F=\Supp(N_2)$
and $N_1$ is a prime divisor.
\end{sclaim}
\begin{proof}
As above, write $\tilde \MMM\sim \tilde N_1+\tilde N_2+\delta E$, where $\delta\ge 0$.
Hence, $\hat \MMM\sim \bar f_*\bar N_1+\bar f_*\bar N_2+\delta\hat E$.
By our assumption $f_*\bar N_1\neq 0$.
Since the class of $\hat \MMM$ is not divisible in $\Cl(\hat X)$,
we have $\delta=0$, $\bar f_*\bar N_2=0$, and $\bar N_1$ is a prime divisor.
\end{proof}

\begin{sclaim}
\label{claim:subsystem}
If $a>1$, then $\MMM\sim N+\delta F$, where $\delta$ is a positive integer and 
$N$ is an effective Weil divisor such that $\dim |N|\ge 2$.
\end{sclaim}
\begin{proof}
By Remark~\ref{sremark:discr}\ $\bar f(\bar F)$ is a point, say $\hat P$.
Let $\hat \MMM_P\subset \hat \MMM$ be the linear subsystem consisting of all divisors from $\hat \MMM$ passing through $\hat P$. Then $\dim (\hat \MMM_P)=2$
and 
\[
\bar \MMM_P\sim \bar f^*\hat \MMM_P-\delta \bar F\sim f^*\hat \MMM-\delta \bar F,
\]
where $\delta\ge 1$. Therefore, $\MMM\sim \MMM_P+\delta F$, where $\dim (\MMM_P)=2$.
\end{proof}

Now we are in position to prove Theorem~\ref{thm:cbQf}\ref{thm:cbQf-a}.
We show that in our assumptions \ref{ass} the diagram \ref{diagram-main-1}
does not exist. Consider possibilities for $\hat X$ case by case.

\subsection{Case $\Cl(X)\simeq \ZZ$}
\label{case:Cl=Z}
Then $e=d$ by Claim ~\ref{claim:tors0}.
Apply Proposition \ref{prop:search1}.
We obtain the following subcases.

\subsubsection*{Subcase $\qQ(X)=7$}
Then by \ref{prop:search1}\ref{prop:search1:q=7} there are prime divisors $M_2$ and $M_3$ 
such that $\MMM\sim 3M_2\sim 2M_3$. But then by Claim~\ref{claim:MN-curve}
we have $M_2=M_3=F$, a contradiction.

\subsubsection*{Subcase $\qQ(X)=5$}
Then $\dim |2A|=0$ and $\dim |3A|= 1$ by \ref{prop:search1}\ref{prop:search1:q=5}.
Hence, $b=4$ and $\MMM\ni 2 D$, where $D\in |2A|$. By Claim~\ref{claim:MN-curve} we have $\Supp(D)=F$, $d=1$ and $a\ge 2$. Then by Claim~\ref{claim:subsystem} we have 
$\dim |kA|\ge 2$, where $k\le 3$, a contradiction. 

\subsubsection*{Subcase $\qQ(X)=4$.}
It follows from Claim~\ref{claim:QDS:eq} that $b=3$ and $ad=2$. 
Thus $\MMM\sim 3A$ and $s_3=1$.
According 
to Proposition \ref{prop:search1}\ref{prop:search1:q=4} we have $\dim |A|\le 0$ and $\dim |2A|\ge 1$.
The relation \eqref{equation-main} for $k=3$ has the form
\begin{equation*}
2=(4\beta_3-3 \alpha) e,
\end{equation*}
where $\beta_3\ge 2\alpha$ by Lemma~\ref{lemma-cthreshold}.
Hence $\alpha\le 2/(5e)<1$ and so $P:=f(E)$ is a non-Gorenstein point of $X$.
Let $r$ be its index. 
Similarly, the relation \eqref{equation-main} for $k=2$ has the form
\begin{equation*}
(\alpha-2\beta_2)e=2(s_2-1).
\end{equation*}
Since $\dim |2A|>0$, a general member of $\MMM_2=|2A|$ is not contracted on $\hat X$, hence $s_2\ge 1$. Therefore, $\beta_2\le \alpha/2$.
Since $r\MMM_2$ is Cartier at $P$, $r\beta_2$ is an integer. 
So, $\alpha\ge 2/r$. Hence $f$ cannot be a Kawamata blowup of $P\in X$ 
and $P\in X$ is not a cyclic quotient singularity
(see Theorem~\ref{theorem-Kawamata-blowup}).
By Proposition \ref{prop:search1}\ref{prop:search1:q=6} we have $r\ge 7$. 
Then 
\[
\frac 1{5e} \ge \frac{\alpha}2 \ge \beta_2 \ge \frac17, \quad e=1. 
\]
By Claim \ref{claim:tors0}\ $d=1$, $a=2$, and $\dim |2A|\ge 2$ by Claim \ref{claim:subsystem}.
Then by \ref{prop:search1}\ref{prop:search1:q=4} we have only one possibility: $A^3=2/11$ and $\B(X)=(11)$.
But this means that 
$P\in X$ is a cyclic quotient singularity, a contradiction.

\subsubsection*{Subcase $\qQ(X)=3$}
Then $b=2$ and $d=a=1$ by Claim~\ref{claim:QDS:eq}. Hence, $\MMM\sim 2F$
because $\Cl(X)\simeq \ZZ$. Then 
we get a contradiction by Claim~\ref{claim:MN-curve}.

\subsection{Case: $\Clt{X}\neq 0$}
Let $n$ be the order of $\Clt{X}$.
Apply Proposition \ref{prop:search2} and consider possibilities for $\qQ(X)$.

\subsubsection*{Subcase $\qQ(X)=5$.}
Then $n=2$.
By 
\ref{prop:search2}\ref{prop:search2:q=5} we have 
\[
\dim |kA|\le 1\quad\text{and} \quad \dim |kA+T|\le 1\quad \text{for all $k\le 3$.}
\]
In particular, $b=4$. If $a>1$, then by Claim~\ref{claim:subsystem} we get a contradiction.
Thus $a=1$ and $d=3$ (see Claim~\ref{claim:QDS:eq}). 
If $\MMM=|4A|$, then we get a contradiction by Claim~\ref{claim:MN-curve}
because $|A|\neq \varnothing$ (see \ref{prop:search2}\ref{prop:search2:q=5}). Let $\MMM=|4A+T|$. Then $\MMM\sim 2N_1+N_2$,
where $N_1\in |A|$ and $N_2\in |2A+T|$ are prime divisors.
By Claim~\ref{claim:MN-new} we have $F=N_1$ and so $d=1$, a contradiction. 

\subsubsection*{Subcase $\qQ(X)=4$} 
Then $n=2$ or $5$.
It follows from Claim~\ref{claim:QDS:eq} that $b=3$ and $ad=2$. 
If $a>1$, we get a contradiction by Claim~\ref{claim:subsystem} and \ref{prop:search2}\ref{prop:search2:q=4}. Thus, $a=1$ and $d=2$.
By \ref{prop:search2}\ref{prop:search2:q=4} $|A+T|\neq\varnothing$ and $\dim |3A+3T|=3$.
Take $\MMM= |3A+3T|$ and apply the construction \eqref{diagram-main}. 
Then by Claim~\ref{claim:MN-curve} $\hat X$ is not a smooth quartic double solid
because $\MMM\ni 3N$, $N\in |A+T|$.
By \eqref{eq:constr:Klambda3} we have $\uplambda(\hat X, \hat \MMM)> \uplambda (X,\MMM)=4/3$. Therefore, $\qQ(\hat X)> 4$.
Hence we can proceed with the sequence of Sarkisov links \eqref{eq-seq}
so that $\uplambda(X^{(i)}, \MMM^{(i)})>4/3$ and $\qQ(X^{(i)})> 4$.
By the above considered cases we get a contradiction.

\subsubsection*{Subcase $\qQ(X)=3$.}
It follows from Claim~\ref{claim:QDS:eq} that $b=2$, $a=d=1$. 
Hence, $\MMM\qq 2F$.
If $\MMM\sim 2F$, we get a contradiction by Claim~\ref{claim:MN-curve}.
Therefore, $T:=\MMM-2F$ is a non-trivial torsion element. We have $K_{\bar X}+2\bar \MMM-\bar F\sim 0$. Hence $-K_X\sim 2\MMM-F\sim 3F+2T$.
Moreover, the group $\Cl(X)$ is generated by the classes of $F$ and $T$.
By Claim ~\ref{claim:tors0} $\hat E\not\sim \hat \MMM$ because $d=1$.

Assume that $3T\sim 0$. Then $\MMM\sim 2(F+2T)$ and $|F+2T|=\varnothing$
by Claim \ref{claim:MN-curve}. This contradicts~\ref{prop:search2}\ref{prop:search2:q=3}.
Therefore, $2T\sim 0$. Hence, $-K_X\sim 3F$.
If $\dim |2F|=3$, then we take $\MMM=|2F|$ and apply the construction \eqref{diagram-main}. 
Then by Claim~\ref{claim:MN-curve} $\hat X$ is not a smooth quartic double solid.
By \eqref{eq:constr:Klambda3} we have $\uplambda(\hat X, \hat \MMM)> \uplambda (X,\MMM)=3/2$. Therefore, $\qQ(\hat X)> 3$.
Hence we can proceed with the sequence of Sarkisov links \eqref{eq-seq}
so that $\uplambda(X^{(i)}, \MMM^{(i)})>3/2$ and $\qQ(X^{(i)})> 3$.
By the above considered cases we get a contradiction.

Therefore, $\dim |2F|<3$.
Thus by \ref{prop:search2}\ref{prop:search2:q=3} we have $\B=(2,4,14)$.
Then $|F+T|\neq \varnothing$ and for $D\in |F+T|$ we have
$-K_X\sim \MMM+D$,
\[
K_{\tilde X}+\tilde \MMM +\tilde D+\gamma E=f^*(K_X+\MMM+D)\sim 0.
\]
where $\gamma\ge \beta-\alpha=\alpha(c-1) \ge 0$.
Since $D\neq F$, taking pushforward to $\hat X$ we obtain $-K_{\hat X}\sim \hat \MMM +\hat D+\gamma \hat E$, where $\hat D\neq 0$. This is possible only if $\gamma =0$.
Then the pair $(X,\MMM)$ is canonical.
In this case by Lemma \ref{lemma:sing}\ref{lemma:sing:c} near each singular point either $D$ or $\MMM$ is Cartier. Since $-4K_X\sim 12D\sim 6\MMM$, the divisor $4K_X$ is Cartier. This contradicts $\B=(2,4,14)$.

\section{Applications}
In this section we consider applications of Theorem~\ref{thm:cbQf} to the existence of conic bundle structures on rationally connected threefolds. 

\begin{proposition}
\label{prop:cb}
Let $X$ be a normal projective threefold and let $\MMM$ be a non-empty linear system of Weil divisors on $X$ such 
that 
\begin{enumerate}
\item\label{prop:cb2}
$\MMM$ is not composed of a pencil, 
\item\label{prop:cb3}
$-(K_X+\MMM)\qq \Theta$, where $\Theta$ is a $\QQ$-divisor which is $\QQ$-Cartier, and
\item\label{prop:cb4}
either $\Theta>0$ or $\Theta=0$ and the pair $(X,\, \MMM)$ is not canonical.
\end{enumerate}
Assume that the variety $X$ has no conic bundle structures.
Then 
there exists a birational transformation $\Psi:X \dashrightarrow X'$, where
$X'$ is del Pezzo threefold as in \xref{thm:cbQf}\ref{thm:cbQf-b}.
Moreover, 
\[
\dim(\MMM)\le \dd(X')+1.
\]
\end{proposition}

\begin{proof}
We may assume that $\MMM$ has no fixed components.
Let 
\[
f:(\tilde X,\tilde \MMM)\longrightarrow (X,\MMM)
\]
be a terminal $\QQ$-factorial 
model of the pair $(X,\MMM)$ (see Corollary~\ref{cor:terminal-model}). Thus $\tilde \MMM:=f^{-1}_*\MMM$, the pair 
$(\tilde X,\tilde \MMM)$ is terminal, $K_{\tilde X}+\tilde \MMM$ 
is $f$-nef, and the variety $\tilde X$ is $\QQ$-factorial. We can write 
\[
K_{\tilde X}+\tilde \MMM\qq f^*(K_X+\MMM)-E',
\]
where $E'$ is the exceptional $\QQ$-divisor and $E'\ge 0$.
Furthermore,
\[
f^*\Theta=\tilde \Theta+ E'',
\]
where $\tilde \Theta$ is the proper transform of $\Theta$ on $\tilde X$
and $E''\ge 0$.
Hence,
\[
K_{\tilde X}+\tilde \MMM +\tilde \Xi\qq f^*(K_X+\MMM+\Theta)\qq 0,
\]
where $\tilde \Xi:=\tilde \Theta+E'+E''\ge 0$. Moreover, $\tilde \Xi>0$ by our assumptions \ref{prop:cb}\ref{prop:cb4}.
In particular, $K_{\tilde X}+\tilde \MMM$ is not nef. 
Run $(K_{\tilde X}+\tilde \MMM)$-MMP. The divisor $\tilde\Xi$ cannot be contracted, so at the end we get a Mori fiber 
space $(\bar X,\bar \MMM)/Z$:
\[
\xymatrix@R=1em@C=5em{
(\tilde X,\tilde \MMM)\ar@{-->}[r] & (\bar X,\bar 
\MMM)\ar[d]^{\varphi}
\\
&Z
}
\]
where $\bar\Xi\qq-(K_{\bar X}+\bar \MMM)$ if $\varphi$-ample.
If $\dim(Z)=2$, then $\varphi$ is a $\QQ$-conic bundle and then we are done.
Assume that $\dim(Z)=1$. Then $\varphi$ is a del Pezzo fibration and $Z\simeq 
\PP^1$.
Since $\dim(\MMM)\ge 2$ and a general member of $\MMM$ is irreducible, $\MMM$ 
is not $\varphi$-horizontal, i.e. $\MMM$ is not a pull-back of a linear system 
on $Z$. 
For a general fiber $F$ by the adjunction formula we have 
\[
K_F=K_{\bar X}|_F=\bar \MMM|_F+\bar\Xi|_F, 
\]
where both $\bar \MMM|_F$ and $\bar\Xi|_F$ are ample. Since $\bar \MMM|_F$ is an (integral) Cartier divisor,
the surface $F$ is either $\PP^2$ or $\PP^1\times \PP^1$. In this case,
$\bar X$ is rational and so it has a lot of conic bundle structures. 
Finally, assume that $Z$ is a point. Then $\bar X$ is a $\QQ$-Fano threefold 
such that $\df(\bar X)\ge 2$ and
\[
K_{\bar X}+\bar \MMM +\bar\Xi\qq 0,
\]
where both $\bar \MMM$ and $\bar\Xi$ are ample, and $\dim (\bar \MMM)\ge 2$. 
Thus we can apply Theorem~\ref{thm:cbQf} to $(\bar X,\bar \MMM)$.
\end{proof}

Now we consider applications of Theorem~\ref{thm:cbQf} to $\QQ$-Fano threefolds.
For simplicity we consider only $\QQ$-Fanos whose Weil divisor class group $\Cl(X)$ has no torsions, i.e. $\Cl(X)\simeq \ZZ$. The collection of invariants 
$\left(\B(X), \qW(X), A^3\right)$ determines the Hilbert series 
\[
\Prm_{(X,A)}(t):= \sum_{n=0}^\infty \dim H^0(X,\OOO_X(nA))\cdot t^n.
\]
An abstract collection $\left(\B, q, \alpha\right)$, where $\B$ is a basket of terminal singularities, $q \in \ZZ_{>0}$, and $\alpha \in \QQ_{>0}$ is called a \textit{numerical candidate} if there are no numerical obstructions 
(like orbifold Riemann-Roch theorem \cite{Reid:YPG} and Bogomolov-Miyaoka inequality \cite{Kawamata:Div-contr}) for
existence of a $\QQ$-Fano threefold with corresponding invariants.

Computer search by using algorithm described in 
\cite{Suzuki-2004}, \cite{Brown-Suzuki-2007j}, \cite{P:2010:QFano}, \cite{P:2019:rat:Q-Fano} (see also \cite{GRD}) shows that 
there are at most 472 numerical candidates with $q\ge 3$. Among them Theorem~\ref{thm:cbQf} is applicable in 313 cases.
Similarly, there are at most 1 492 numerical candidates with $q=2$ and Theorem~\ref{thm:cbQf} is applicable in 382 cases.
We expect that in most of the remaining cases 
the corresponding $\QQ$-Fano threefolds (if they exist)
should have conic bundle structure. 
However proof of this needs a case by case considerations.
Let us consider just one example.

\begin{proposition}
Let $X$ be a $\QQ$-Fano threefold with $\qQ(X)=7$ and $A^3= 1/78$ (No. \textup{41478} in \cite{GRD}).
Then 
$X$ is rational. In particular, $X$ has a conic bundle structure.
\end{proposition}
Note however that the existence of a $\QQ$-Fano threefold $X$ with these 
invariants is not known.
\begin{proof}
In this case the group $\Cl(X)$ is torsion free, 
$\B(X)=(2, 3, 13)$, $\dim |6A|=1$, and $\dim |kA|=0$ for $1\le k\le 5$.
Apply construction \eqref{diagram-main} with $\MMM=|6A|$. Then near the point of index $13$ we have 
$\MMM\sim 12(-K_X)$. Hence, $\beta_6\ge 12\alpha$ by 
Lemma~\ref{lemma-cthreshold} and so 
\[
\beta_1\ge \beta_6/6\ge 2\alpha. 
\]
The 
relation \eqref{equation-main} for $k=1$ has the form
\begin{equation} 
\label{equation-main:2-3-13}
\hat{q}=7 s_1+(7\beta_1- \alpha) e\ge 7 s_1+13\alpha e.
\end{equation}
Since $\hat q\le 7$, we have $s_1=0$ and $\alpha<1$.
Thus $f$ is the Kawamata blowup of a cyclic quotient singularity $P\in X$ of index $r=2$, $3$ or $13$. 
If $r=2$ or $3$, then $A\sim -K_X$ near $P$. Hence $\beta_1=\alpha+m_1$, where 
$m_1$ is an integer, $m_1\ge \alpha>0$. 
Then \eqref{equation-main:2-3-13} gives us
\begin{equation*} 
\hat{q}= 6\alpha e+ 7m_1e>7.
\end{equation*}
This contradicts our assumptions. Thus $r=13$ and so $A\sim 2(-K_X)$ near 
$P$. Hence $\beta_1=2\alpha+m_1$, where $m_1$ is a non-negative integer. Again, 
from \eqref{equation-main:2-3-13} we obtain
\begin{equation*} 
\hat{q}= 13\alpha e+ 7m_1e.
\end{equation*}
Thus $m_1=0$, $\beta_1=2\alpha$, and $\beta_6=12\alpha$. 
By Theorem~\ref{theorem-Kawamata-blowup}\ $\alpha=1/13$.
Now, the relation \eqref{equation-main} for $k=6$ has the form:
\begin{equation*} 
6 \hat{q}=7 s_6+(7\beta_6-6 \alpha) e= 7 s_6+6e.
\end{equation*}
Since $\hat q<8$, $s_6=0$. Since $\dim |6A|=1$, this implies that the contraction $\bar f$ is not birational, so it is a del Pezzo fibration. 
Since $s_1=0$, we have $\bar \MMM\sim 6\bar M_1$ and so $\bar M_1$ is a fiber of multiplicity $6$. By the main result of \cite{MP:DP-e} the general fiber of $\bar f$ is a del Pezzo surface of degree $6$. But then $\bar X$ must be rational, a contradiction.
\end{proof}

 \newcommand{\etalchar}[1]{$^{#1}$}
\def\cprime{$'$}


\end{document}